\documentclass[a4paper,11pt,reqno,noindent]{amsart}
\usepackage[centertags]{amsmath} 
\usepackage{amsfonts,amssymb,amsthm}
\usepackage{hyperref}
\usepackage{mathtools}

 \hypersetup{pdfmenubar=false, % show Acrobat¡¯s menu?  
pdfnewwindow=true, %links in new window 
colorlinks=false, % false: boxed links; true: colored links
linkcolor=blue, % color of internal links 
citecolor=blue, % color of links to bibliography 
filecolor=magenta, % color of file links 
urlcolor=cyan % color ofexternal links 
} 
\usepackage{graphicx}

\usepackage[english]{babel} \usepackage{newlfont} \usepackage{color}
\usepackage[body={15cm,21.5cm},centering]{geometry} \usepackage{fancyhdr}
 \usepackage{esint} \usepackage{enumerate}

\usepackage[active]{srcltx}

\newtheorem{theorem}{Theorem} \newtheorem{lemma}{Lemma}
\newtheorem{proposition}{Proposition} 
\newtheorem{corollary}{Corollary} 
\newtheorem{definition}{Definition}

\theoremstyle{definition} 
 
\newtheorem{remark}{Remark}

 \newcommand{\Lip}{\mathrm{Lip}\,}
\newcommand{\loc}{\mathrm{loc}}

\newcommand{\dist}{\operatorname{dist}} \newcommand{\supp}{\operatorname{supp}}

\newcommand{\e}{\varepsilon}

\newcommand{\R}{\mathbb{R}}

\renewcommand{\d}{\mathrm{d}} \renewcommand{\L}{\mathbb{L}}

 \newcommand\sgn{\mathrm{sgn}}

 \renewcommand{\H}{\mathcal{H}}
 
\renewcommand{\L}{{\mathcal L}}

\renewcommand{\div}{\mathrm{div}\,}

\newcommand{\cof}{\mathrm{cof}\,}
\renewcommand{\j}{\mathbf{j}}

\newcommand{\eps}{\varepsilon}

\newcommand{\Lm}{\mathcal{L}}%Lebesgue measure
\newcommand{\crit}{\frac{n-1}{n}}

\title{Coarea formulae and chain rules for the Jacobian determinant in fractional Sobolev spaces} \date{\today} 
\author[P. Gladbach]{Peter Gladbach}
\author[H. Olbermann] {Heiner Olbermann} 

\begin{document}
\maketitle

\begin{abstract}
We prove weak and strong versions of the coarea formula and the chain rule for
distributional Jacobian determinants $Ju$ for functions $u$ in  fractional Sobolev
spaces $W^{s,p}(\Omega)$, where $\Omega$ is a bounded domain in $\R^n$ with
 smooth boundary. The weak
forms of the formulae are proved for the range $sp>n-1$, $s> \crit$,
while the strong versions are proved for the range $sp\geq n$, $s\geq \frac{n}{n+1}$. We also provide a chain rule for distributional Jacobian determinants of H\"older functions and point out its relation to two open problems in geometric analysis.
\end{abstract}

\section{Setting and statement of main result}

Let $n\geq 2$, and $\Omega\subset\R^n$ bounded and open with smooth boundary. The fractional Sobolev space
$W^{s,p}(\Omega)$ with $s\in (0,1)$, $p\in [1,\infty)$ is defined as the set
of functions $u\in L^p(\Omega)$ such that
\[
[u]_{W^{s,p}}^p \coloneqq \int_{\Omega\times\Omega}\frac{|u(x)-u(y)|^p}{|x-y|^{n+sp}}\d x\d
y<\infty\,.
\]
The norm on $W^{s,p}(\Omega)$ is defined by $\|u\|_{W^{s,p}}=\|u\|_{L^p}+[u]_{W^{s,p}}$. For vector-valued functions $W^{s,p}(\Omega;\R^m)$ is defined componentwise.

%Secondly, we may define them as  real interpolation spaces,
% \[
% W^{s,p}(\Omega)=[L^p(\Omega),W^{1,p}(\Omega)]_{s,p}\,.
% \]
% For a detailed definition of interpolation spaces, we refer to Section \ref{sec:prerequisites}

Here we are concerned with the \emph{distributional Jacobian} $Ju$ for $u\in
W^{\crit,n}(\Omega;\R^n)$. For smooth $u$, this may be
defined as follows: We set
\[
\j u(x) \coloneqq \frac{1}{n}(\cof \nabla u(x))^Tu(x)\,,
\]
where $\cof \nabla u(x)\in \R^{n\times n}$ denotes the cofactor matrix of $\nabla u(x)$. 
  Then we see that $\det \nabla u = \div \j u$, so that we have
\[
  \begin{split}
    \left<J u,\psi\right>&\coloneqq \int_\Omega\det\nabla u (x)\psi(x)\d x\\
    &=\int_\Omega \div \j u(x) \psi(x) \d x\\
    &= -\int_\Omega
    \j u(x)\cdot\nabla \psi(x)\d x
  \end{split}
\]
for $\psi \in C^1_c(\Omega)$.
For $u\in
L^\infty\cap W^{1,n-1}(\Omega;\R^n)$, we may obviously define $\j u$ as an
object in 
$L^1(\Omega;\R^n)$, and $Ju$ is defined distributionally via $Ju=\div \j u$. In
this paper, we follow \cite{MR2810795} to obtain a well-defined notion of $\j u\in L^1(\Omega;\R^n)$ for $u$ in $W^{\crit,n}(\Omega;\R^n)$ through multilinear interpolation, see Section \ref{sec:prerequisites}.
The thus defined distribution $Ju=\div \j u$  is our object of investigation; we are going to derive coarea type formulae
and chain rules for it, supposing  $u$ belongs to  a certain family of subspaces of $W^{\crit,n}(\Omega;\R^n)$. 

To be more precise, let
$u^a(x) \coloneqq \frac{u(x)-a}{|u(x)-a|}$ for  $a\in \R^n$. We will prove the following theorems:
\begin{theorem}[Weak coarea formula and chain rule]\label{theorem: weak coarea chain}
  Let  $s\in( \crit,1)$,  $sp\in (n-1,\infty]$, and $u\in W^{s,p}(\Omega;\R^n)$. Then for almost every $a\in \R^n$, we
  have $u^a\in W^{\crit,n}(\Omega;\R^n)$ and the following two statements hold:
  \begin{itemize}
  \item[(i)] (Weak coarea formula) For every $\psi\in C^1_c(\Omega)$, we have
    \[
    \left<Ju,\psi\right>=\frac{1}{\omega_n}\int_{\R^n}\left<Ju^a,\psi\right>\d
    a\,.
    \]
\item[(ii)]  (Weak chain rule)  For every $F\in C^1(\R^n;\R^n)$, and every $\psi\in C^1_c(\Omega)$, we have
    \[
    \left<J(F\circ u),\psi\right>=\frac{1}{\omega_n}\int_{\R^n}\det \nabla F(a) \left<Ju^a,\psi\right>\d a\,.
    \]
  \end{itemize}
  
  Here $\omega_n = \Lm^n(B(0,1))$ is the volume of the $n$-dimensional unit ball.
\end{theorem}

\begin{remark}
The validity of the weak coarea formula and chain rule for the critical case
$u\in W^{\crit,n}(\Omega;\R^n)$ is not treated here and remains an open question. % and $u\in C^0\cap
     % W^{{n}/(n+1),n+1}(\Omega;\R^n)$ 
     % in Theorem \ref{thm:strong}
     %  We believe that it should hold true for
     % this case as well.

\end{remark}

In our second theorem we are going to use the following notation: For  a  distribution $T:C^1_c(\Omega)\to\R$,
we denote its total variation by
\[
  |T|_{TV}\coloneqq\sup\{\left<T,\psi\right>:\psi\in C^1_c(\Omega),\,|\psi|\leq 1\}\,.
\]
By the
Riesz-Radon representation theorem, $T$ can
be extended to a Radon measure on $\Omega$ if $|T|_{TV}<\infty$.

\begin{theorem}[Strong coarea formula and chain rule]
\label{thm:strong}
  Let $u\in C^0\cap W^{\frac{n}{n+1},n+1}(\Omega;\R^n)$. Assume
  that $|Ju|_{TV}< \infty$ . Then the following two statements hold:
  \begin{itemize}
\item[(i)] (Strong coarea formula.)  For almost every $a\in\R^n$, $Ju^a$ can
  be extended to a 
  Radon measure and 
    \[
    \left|J
      u\right|_{TV}=\frac{1}{\omega_n}\int_{\R^n} 
\left|Ju^a\right|_{TV}\d a\,.
    \]
  \item[(ii)](Strong chain rule.) Let $F\in C^\infty(\R^n;\R^n)$ be globally
    Lipschitz. Then ${|J(F\circ u)|_{TV} < \infty}$, and for every $\psi\in C^1_c(\Omega)$, we have
    that 
    \[
    \left<J(F\circ u),\psi\right>=\int_\Omega \det \nabla F(u(x))\psi(x)\d Ju(x)\,.
    \]
  \end{itemize}
\end{theorem}

Weak and strong versions of the coarea formulae and chain rule for the Jacobians of (non-fractional) Sobolev functions have been treated before in \cite{MR1911049,de2003some}.

\medskip

Finally, in Section \ref{sec:chain-rules-holder} below,  we want to point out
the relevance of chain rules for two well-known open problems in geometric
analysis: the $C^{1,\alpha}$ isometric immersion problem and the H\"older
embedding problem for the Heisenberg group. In order to avoid any
misunderstanding, let us state clearly that we do not have any new results on
these problems, but only offer a new vantage point. In order to do so, we
formulate a chain rule for H\"older functions $u$, which allows for a slightly
larger domain of $Ju$. In particular, we may define it on  characteristic functions $\chi_E$ of sets of finite perimeter $E$ such that $\overline E\subset \Omega$.

\begin{theorem}
  \label{thm:hoelder}
  Let $\alpha\in (\crit,1)$, $u\in C^{0,\alpha}(\Omega;\R^n)$, $F\in C^1(\R^n;\R^n)$ globally Lipschitz and $E$ as above. Then
    \[
    \left<J(F\circ u),\chi_E\right>=
    \frac{1}{\omega_n}\int_{\R^n}\det\nabla F(a)\left<Ju^a,\chi_E\right>\d a\,.
  \]
  \end{theorem}

  \subsection*{Plan of the paper}
  In Section \ref{sec:prerequisites}, we prove the estimate that defines the distributional Jacobian $Ju$ for $u\in W^{\crit,n}(\Omega;\R^n)$. Theorem \ref{thm:strong} and \ref{thm:hoelder} will be fairly straightforward consequences of these estimates, and they will be proved in Sections \ref{sec:proof-strong-chain} and \ref{sec:proof-chain-rule-hoelder} respectively. The proof of the weak chain rule (Theorem \ref{theorem: weak coarea chain}) is the most interesting, and will be carried out first,  in Section \ref{sec:proof-weak-chain}. In Section \ref{sec:two-open-questions}, we will explain the relation between chain rules and the  two open problems in geometric analysis mentioned above.

  \subsection*{Some notation}
The symbol $D$ denotes the distributional derivative, while $\nabla=(\partial_1,\dots,\partial_n)$ is the approximate gradient.
The function spaces $W^{s,p}(\Omega;\R^n)$ will be abbreviated by $W^{s,p}$ when it is clear from the context what is meant. 
  The  H\"older spaces $C^{0,\alpha}(\Omega;\R^m)$ are defined through the norm
\[
  \|u\|_{C^{0,\alpha}} \coloneqq \|u\|_{C^0}+\sup_{x\neq y}\frac{|u(x)-u(y)|}{|x-y|^\alpha}\,\,,
    \]
and $C^{0,\alpha}(\Omega;\R^m)\coloneqq \{u\in C^0(\Omega;\R^m):\|u\|_{C^{0,\alpha}}<\infty\}$. The space $BV(\Omega)$ is the set of all functions $u\in L^1(\Omega)$ such that
\[
  \sup\left\{  \int_\Omega u(x)\div\psi(x)\d x:\psi\in C^1_c(\Omega),\,|\psi|\leq 1\right\}<\infty\,.
\]
For $u\in BV(\Omega)$, the distributional derivative $Du$ is understood as a matrix-valued Radon measure.
The symbol ``$C$'' is used as follows: A statement such as  $f\leq C(\alpha)g$ is to be read as ``there exists a constant $C>0$ that depends only on $\alpha$ such that $f\leq C g$''. For $f\leq Cg$, we also write $f\lesssim g$.  The $k$-dimensional Hausdorff measure is denoted by $\H^k$. The Brouwer degree of  $u\in C^0(\Omega;\R^n)$ in $a\in\R^n\setminus u(\partial\Omega)$ is denoted by $\deg(u,\Omega,a)$.

\section{Estimate defining the distributional Jacobian}

\label{sec:prerequisites}
% Real interpolation spaces can be expressed in several equivalent ways. In our
% case, the most convenient one is the definition of real interpolation via trace
% spaces. I.e., given two Banach spaces $X_0,X_1$ that are both embedded in a
% larger Banach space $\bar X$, we define $[X_0,X_1]_{\theta,p}$ as follows. For a
% function $U:\R^+\to X_0\cap X_1$, let 
% $U_\theta(t):=t^{-\theta}U(t)$, and for an interval $I\subset\R_+=(0,\infty)$, let $L^p_*(I;X)$ be the space of $X$-valued $L^p$ functions with respect to the measure $\d t/t$. Then 
% we define $[X_0,X_1]_{\theta,p}$ as the space of $u\in X_0+X_1$  for which the norm 
% \[
% \begin{split}
%   \|u\|_{[X_0,X_1]_{\theta,p}}&=\inf\Bigg\{\|U_\theta\|_{L^p_*(\R^+,X_1)}
%  + \left\|(U')_\theta\right\|_{L^p_*(\R^+,X_0)}:
%  \\
%  &\qquad U\in
%      C^1(\R^+;X_0\cap X_1),\,\lim_{t\to 0} U(t)=u\Bigg\}\,,
%   \end{split}
% \]
% is finite.
% This definition is equivalent to the more well-known $K$-method, see \cite{lunardi2018interpolation}.

% \medskip

% Specifying to fractional Sobolev spaces defined over a bounded domain $\Omega\subset\R^n$ with Lipschitz boundary, we have the definition
% \[
%   W^{s,p}(\Omega)=\left[ L^p(\Omega),W^{1,p}(\Omega)\right]_{s,p}\,.
% \]
% The space $W^{s,p}(\Omega)$ may also be defined as the  space of measurable functions for which  the Gagliardo norm
% \[
%   \|u\|_{W^{s,p}(\Omega)}=\int_{\Omega\times\Omega}\frac{|u(x)-u(y)|^p}{|x-y|^{n+sp}}\d x\d y
%   \]
%   is finite. The two norms mentioned above are equivalent (see again \cite{lunardi2018interpolation}).

%   \medskip

We define the distributional Jacobian $J u$ as in \cite{MR2810795}.
This is equivalent to a definition through multilinear interpolation by the method of trace spaces (see e.g.~Chapter 28 of \cite{MR2328004}).

\medskip

For smooth functions $\j u=\frac1n (\cof\nabla u)^T u$ has been defined above, and $Ju = \div \j u$. The following identity will be useful in the sequel: For $u\in C^1(\Omega;\R^n)$, $\psi\in C^1(\Omega)$, we have
  \begin{equation}\label{eq:8}
  \sum_{j=1}^n (\cof\nabla u)_{ij}\partial_{x_j} \psi=\det\left(\nabla u_1,\dots,\nabla u_{i-1},\nabla \psi,\nabla u_{i+1},\dots,\nabla u_n\right)\,.
\end{equation}

\medskip

The following proposition is a statement of three estimates for $\j u$, the
first of which can be found as  Lemma 4 in \cite{MR2810795}; we use the same method of proof. 

    \begin{proposition}
      \label{prop:distrib}
  Let $\Omega\subset \R^n$ be open bounded with Lipschitz boundary, and $\alpha\in\left(\crit,1\right)$. Then for all $u,v\in C^1(\Omega;\R^n)$ and every $\psi\in C^1_c(\Omega)$, we have that
  \[
    \begin{split}
      \Big| \Big<Ju-Jv,\psi\Big>&\Big|\\
      \lesssim\min\Bigg(& \|u-v\|_{W^{\frac{n-1}{n},n}}\left(\|u\|_{W^{\frac{n-1}{n},n}}+\|v\|_{W^{\frac{n-1}{n},n}}\right)^{n-1}\|\nabla\psi\|_{L^\infty}\,,\\
& \|u-v\|_{W^{\frac{n}{n+1},n+1}}\left(\|u\|_{W^{\frac{n}{n+1},n+1}}+\|v\|_{W^{\frac{n}{n+1},n+1}}\right)^{n-1}\|\psi\|_{W^{\frac{n}{n+1},n+1}}\,,\\
      & \|u-v\|_{C^{0,\alpha}}\left(\|u\|_{C^{0,\alpha}}+\|v\|_{C^{0,\alpha}}\right)^{n-1}\|\nabla\psi\|_{L^1}\Bigg)\,.
    \end{split}
    \]

\end{proposition}
  \begin{proof}
    Let $\tilde u,\tilde v$ be extensions of $u,v$ to $\R^n$ such that
    \[
      \begin{split}
      \|\tilde u\|_{W^{\frac{n-1}{n},n}(\R^n;\R^n)}&\lesssim \|u\|_{W^{\frac{n-1}{n},n}(\Omega;\R^n)}\\
      \|\tilde u\|_{W^{\frac{n}{n+1},n+1}(\R^n;\R^n)}&\lesssim \|u\|_{W^{\frac{n}{n+1},n+1}(\Omega;\R^n)}\\
      \|\tilde u\|_{C^{0,\alpha}(\R^n;\R^n)}&\lesssim \|u\|_{C^{0,\alpha}(\Omega;\R^n)}\,,
    \end{split}
  \]
  with analogous estimates for $\tilde v$.
Let $\eta\in C^\infty_c(\R^n)$ be a  standard mollifier, i.e., $\int_{\R^n}\eta(x)\d x=1$, and set $\eta_t \coloneqq t^{-n}\eta
(\cdot/t)$.  
   We may define extensions  $U,V:\Omega\times [0,1)\to\R^n$ of $u,v$  by setting
    \[
      U(x,t) \coloneqq (\eta_t* \tilde u)(x)\,,
\]
again with an analogous definition for $V$. We write $\tilde\nabla=(\nabla,\partial_t)$.
By well known trace estimates (see e.g.~\cite{lunardi2018interpolation,Lions1963}), the  definition above implies
  \begin{equation}
    \label{eq:10}
    \begin{split}
  \|U\|_{W^{1,n}(\Omega\times [0,1);\R^n)}&\lesssim \|u\|_{W^{\frac{n-1}{n},n}}(\Omega;\R^n)\\
  \|U\|_{W^{1,n+1}(\Omega\times [0,1);\R^n)}&\lesssim \|u\|_{W^{\frac{n}{n+1},n+1}(\Omega;\R^n)}\\
\|\tilde\nabla U(\cdot,t)\|_{L^\infty(\Omega;\R^n)}&\lesssim t^{\alpha - 1}\|u\|_{C^{0,\alpha}(\Omega;\R^n)}\,,
\end{split}
\end{equation}
with analogous estimates for $V$.
Furthermore we may extend $\psi$ to a function ${\Psi\in C^1_c(\Omega\times[0,1))}$ such that
  \begin{equation}
    \label{eq:11}
  \begin{split}
  \| \Psi\|_{C^1((\Omega\times[0,1))}&\lesssim\|\psi\|_{C^1(\Omega)}\\
  \| \Psi\|_{W^{1,n+1}(\Omega\times[0,1))}&\lesssim \| \psi\|_{W^{\frac{n}{n+1},n+1}(\Omega)}\\
    \sup_{t\in(0,1)}\|\tilde\nabla \Psi(\cdot,t)\|_{L^1(\Omega)}&\lesssim    \|\nabla \psi\|_{L^1(\Omega)}\,.
  \end{split}
\end{equation}

  Now we have, writing $JU = J(U(\cdot, t))$, that
  
  \begin{equation}\label{eq:9}
    \begin{split}
      \left<Ju-Jv,\psi\right>&=-\int_0^1\partial_t \left<JU-JV,\Psi\right>\d t\\
     & =\int_0^1\int_{\Omega}\partial_t\left( \left(\j U-\j V\right)\cdot \nabla \Psi \right) \d x\d t
    \end{split}
  \end{equation}

  We will now rewrite the expression $\partial_t \left((\j U-\j V)\cdot\nabla \Psi\right) \d x$ using differential forms,
and  claim that it can be written as a sum of
exact forms plus terms that can each be written as a product of a one-homogeneous
function in $\tilde\nabla U-\tilde \nabla V$ with an $n$-homogeneous function
in $(\tilde\nabla U,\tilde \nabla V,\tilde \nabla \Psi)$.
Indeed, with   $ U=(U_1,\dots,U_n)$ and $\d U_i=\sum_{j=1}^n \partial_j U_i\d x_j$
       we have by \eqref{eq:8} that
      \[
          \
          \begin{split}
\j U\cdot\nabla \Psi\d x_1\wedge\dots \wedge \d x_n&=
\frac{1}{n}  
\sum_{i=1}^n U_i\d U_1\wedge\dots\wedge\d U_{i-1}\wedge\d\Psi\wedge\d U_{i+1}\wedge\dots\wedge\d U_n\\
& =:\frac{1}{n}  
\sum_{i=1}^nF_i(U,\Psi)\,.
\end{split}
\]
By the product rule, and denoting derivatives with respect to $t$ by a prime, we get 

\begin{equation}
  \label{eq:6}
  \begin{split}
  \partial_t F_i(U,\Psi)
  &=U_i'\d U_1\wedge\dots\wedge\d U_{i-1}\wedge\d\Psi\wedge\d U_{i+1}\wedge\dots\wedge\d U_n\\
&\quad+\sum_{j\neq i} G_{ij}(U,\Psi)\\
  &\quad+U_i\d U_1\wedge\dots\wedge\d U_{i-1}\wedge\d\Psi'\wedge\d U_{i+1}\wedge\dots\wedge\d U_n\,,
\end{split}
\end{equation}
  where
  \[
    G_{ij}(U,\Psi)=
    U_i\d U_1\wedge\dots\wedge\d U_j'\wedge\dots\wedge\d U_{i-1}\wedge\d\Psi\wedge\d U_{i+1}\wedge\dots\wedge\d U_n\,.
\]  
We treat the terms $G_{ij}$ with an integration by parts,
  \[
    \begin{split}
    G_{ij}(U,\Psi)&=(-1)^{j-1}\d\left(U_iU_j'\d U_1\wedge\dots\wedge\d U_{j-1}\wedge\d U_{j+1}\wedge\dots\wedge\d U_{i-1}\wedge\d\Psi\wedge\d U_{i+1}\wedge\dots\wedge\d U_n\right)\\
    &\quad+U_j'\d U_1\wedge\dots\wedge\d U_{j-1}\wedge\d\Psi\wedge\d U_{j+1}\wedge\dots\wedge\d U_n\,\,.
  \end{split}
\]
In the same way,
\[
  \begin{split}
  U_i\d U_1\wedge&\dots\wedge\d U_{i-1}\wedge\d U_{i+1}\wedge\dots\wedge\d U_n\wedge\d\Psi'\\
  &=(-1)^{i-1} \d \left(U_i\Psi'\d U_1\wedge\dots\wedge\d U_{i-1}\wedge\d U_{i+1}\wedge\dots\wedge\d U_n\right)\\
  &\quad+ \Psi'\d U_1\wedge\dots\wedge\d U_n\,.
\end{split}
  \]
  Inserting these last two equations in \eqref{eq:6}, we obtain
  \[
    \begin{split}
 \partial_t\j U\cdot \nabla\Psi \d x &=\sum_{i=1}^n U_i'\d U_1\wedge \dots\wedge \d U_{i-1}\wedge\d\Psi\wedge\d U_{i+1}\wedge\dots\wedge\d U_n\\
    &\quad +\Psi'\d U_1\wedge\dots\wedge\d U_n\\
    &\quad+ \d H(U,\Psi)\,,
  \end{split}
\]
where $\d H(U,\Psi)$ is the sum of all the exact forms that appeared in the preceding calculations. % We note that the integral over $\d H_i$  will disappear when inserted in \eqref{eq:9}.

\medskip

We turn to the computation of $\partial_t(\j U-\j V)\cdot\nabla\Psi$. We have that
\[
  \begin{split}
  U_i'\d U_1\wedge& \dots\wedge \d U_{i-1}\wedge\d\Psi\wedge\d U_{i+1}\wedge\dots\wedge\d U_n\\
 &\quad -V_i'\d V_1\wedge \dots\wedge \d V_{i-1}\wedge\d\Psi\wedge\d V_{i+1}\wedge\dots\wedge\d V_n\\
 &=(U_i'-V_i')\d U_1\wedge \dots\wedge \d U_{i-1}\wedge\d\Psi\wedge\d U_{i+1}\wedge\dots\wedge\d U_n\\
 &\quad  +\sum_{j\neq i}  V_i'\d V_1\wedge \dots\wedge\d V_{j-1}\wedge(\d U_j-\d V_j)\wedge\\
 &\qquad\qquad \wedge\d U_{j+1}\wedge\dots\wedge \d U_{i-1}\wedge\d\Psi\wedge\d U_{i+1}\wedge\dots\wedge\d U_n
\end{split}
\]
and similarly
\[
  \begin{split}
 \Psi'\d U_1\wedge\dots\wedge\d U_n&-\Psi'\d V_1\wedge\dots\wedge\d V_n\\
 &=\sum_{i=1}^n\Psi'\d V_1\wedge\dots\wedge\d V_{i-1}\wedge(\d U_i-\d V_i)\wedge\d U_{i+1}\wedge\dots\wedge\d U_n\,.
\end{split}
\]
This proves the claim about the structure of $\partial_t(\j U-\j
V)\cdot\nabla\Psi \d x$ that we have made above, and we may summarize by saying
that it  can be written as a sum of exact forms plus terms that can be estimated by
\[
  C  |\tilde\nabla U-\tilde\nabla V| (|\tilde\nabla U|+|\tilde\nabla V|)^{n-1}
  |\tilde \nabla\Psi|\,.
\]
% where $\tilde\nabla$ denotes the $n+1$-dimensional gradient $\tilde\nabla=(\nabla,\partial_t)$.
Inserting this in \eqref{eq:9},  we  obtain the following  three estimates by H\"older's inequality:
\[  
  \begin{split}
  \big| \big<&Ju-Jv,\psi\big>\big|\\
  &\lesssim \min \Big(\|\tilde\nabla\Psi\|_{L^\infty} \|\tilde\nabla U-\tilde\nabla V\|_{L^n}\left(\|\tilde\nabla U\|_{L^n}+\|\tilde\nabla V\|_{L^n}\right)^{n-1}\,,\\
  &\qquad\qquad \|\tilde\nabla\Psi\|_{L^{n+1}} \|\tilde\nabla U-\tilde\nabla V\|_{L^{n+1}}\left(\|\tilde\nabla U\|_{L^{n+1}}+\|\tilde\nabla V\|_{L^{n+1}}\right)^{n-1}\,,  \\
  &\qquad\qquad \int_0^1\|\tilde \nabla \Psi(\cdot,t)\|_{L^1}\|\|\tilde\nabla U(\cdot,t)-\tilde\nabla V(\cdot,t)\|_{L^\infty}\left(\|\tilde\nabla U(\cdot,t)\|_{L^\infty}+\|\tilde\nabla V(\cdot,t)\|_{L^\infty}\right)^{n-1}\d t\Big)
\end{split}
\]
Using  \eqref{eq:10}, \eqref{eq:11} this becomes
\[
  \begin{split}
    \left| \left<Ju-Jv,\psi\right>\right|
      &\lesssim \min \Big(\|\psi\|_{C^1} \|u-v\|_{W^{\frac{n-1}{n},n}}\left(\|u\|_{W^{\frac{n-1}{n},n}}+\|v\|_{W^{\frac{n-1}{n},n}}\right)^{n-1}\,,\\
  &\qquad\qquad \|\psi\|_{W^{\frac{n}{n+1},n+1}} \|u-v\|_{W^{\frac{n}{n+1},n+1}}\left(\|u\|_{W^{\frac{n}{n+1},n+1}}+\|v\|_{W^{\frac{n}{n+1},n+1}}\right)^{n-1}\,,  \\
&\qquad\qquad\int_0^1 t^{n(\alpha-1)}\d t \|\nabla \psi\|_{L^1}\|u-v\|_{C^{0,\alpha}}\left(\|u\|_{C^{0,\alpha}}+\|v\|_{C^{0,\alpha}}\right)^{n-1}\Big)  \,.
  \end{split}
\]
By the assumption that $n(\alpha-1)>-1$, the integral in the last line is finite, proving the lemma.
  \end{proof}

By the proposition, the following definition of $Ju$, which gives rigorous meaning to the statement of the theorems of the preceding section, is well-defined.
  
\begin{definition}
  \label{def:Ju}
    Let $u\in W^{\crit,n}(\Omega;\R^n)$. The distributional Jacobian $Ju:C_c^1(\Omega)\to \R$ is defined by
    \[
      \langle Ju,\psi\rangle \coloneqq \lim_{k\to\infty} \langle Ju_k,\psi\rangle
    \]
    where $u_k$ is any sequence in $C^1(\Omega;\R^n)$ that approximates $u$ in $W^{\crit,n}(\Omega;\R^n)$.
    For $u\in C^{0,\alpha}(\Omega;\R^n)$, $Ju$ may be extended to an element of the dual of $BV(\Omega)$ by
          \[
      \langle Ju,\psi\rangle \coloneqq \lim_{k\to\infty} \langle Ju_k,\psi_k\rangle\,,
    \]
  
  where $u_k$ is a sequence in $C^1(\Omega;\R^n)$ such that $u_k\to u$ in $C^{0,\alpha-\delta}$ for some $\delta>0$ with $\alpha-\delta>\crit$, and $\psi_k$ is a sequence in $C_c^1(\Omega)$ with $\psi_k\to \psi$ weakly-$\ast$ in $BV(\Omega)$, i.e.,
  \[
    \begin{split}
    \psi_k&\to \psi\text{ in }\in L^1(\Omega)\\
    \int_\Omega \varphi(x) |\nabla\psi_k|(x)\d x &\to\int_\Omega \varphi(x) \d|D\psi|(x)\text{ for any  }\varphi \in C_0(\Omega)\,,
  \end{split}
\]
where $C_0(\Omega)$ denotes the closure of $C_c^0(\Omega)$ with respect to the $C^0$-norm.
\end{definition}
  \section{Proof of the weak chain rule}
  \label{sec:proof-weak-chain}

\begin{figure}
\includegraphics{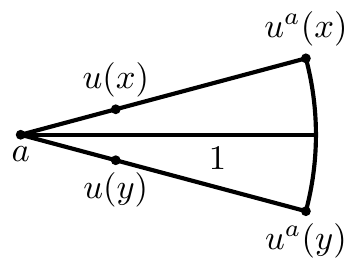}
\caption{For $a\in \R^n$, $u^a:\Omega \to \R^n$ is the stereographic projection of $u$ onto the unit sphere. If $u$ is smooth and $a$ is a regular value of $u$, then $Ju^a = \omega_n \sum_{x\in u^{-1}(a)} \sgn \det \nabla u(x) \delta _x$, see the appendix.}
\end{figure}

\begin{lemma}\label{lemma: u^a continuity}
Let $s\in (\crit,1]$, $sp\in (n-1,\infty]$, and $u_k \to u$ in  $W^{s,p}(\Omega; \R^n)$. Then the functions
\begin{equation*}
a \mapsto\|u_k^a - u^a\|_{W^{\frac{n-1}{n},n}(\Omega;\R^n)}^n
\end{equation*}
converge to $0$ in $L^1_\loc(\R^n)$.
\end{lemma}

\begin{remark}
Note that we require convergence in the strictly smaller space $W^{s,p}(\Omega;\R^n)$ to get convergence of almost every $u_k^a$ to $u^a$ in $W^{\crit,n}(\Omega;\R^n)$, at least for a subsequence.
\end{remark}

\begin{proof}
Pick any subsequence of $u_k$. Withouth relabeling, choose a subsequence such that $u_k \to u$ almost everywhere. Then,  noting that $|u^a(x)| = 1$, by the dominated convergence theorem,
\[
\int_{B(0,R)} \|u_k^a - u^a\|_{L^n(\Omega;\R^n)}^n\d x \to 0\, .
\]

For the nonlocal part of the  $W^{\crit,n}$-norm, we use the estimate
\begin{equation*}
|u^a(x) - u^a(y)| \leq \min \left(\frac{|u(x)-u(y)|}{\min(|u(x)-a|,|u(y)-a|)}, 2\right)
\end{equation*}
in the triple integral
\begin{equation*}
\begin{aligned}
&\int_{B(0,R)} [u^a_k - u^a]_{W^{\frac{n-1}{n},n}}^n \d a \\
= & \int_{B(0,R)} \int_\Omega \int_\Omega \frac{|u_k^a(x) - u_k^a(y) - u^a(x) + u^a(y)|^n}{|x-y|^{2n-1}} \d x \d y \d a\\
= & \underbrace{\int_{B(0,R)} \int\int_{|x-y|>\delta} \frac{|u_k^a(x) - u_k^a(y) - u^a(x) + u^a(y)|^n}{|x-y|^{2n-1}} \d x \d y \d a}_{I_k}\\
 & + \underbrace{\int_{B(0,R)} \int\int_{|x-y| \leq \delta} \frac{|u_k^a(x) - u_k^a(y) - u^a(x) + u^a(y)|^n}{|x-y|^{2n-1}} \d x \d y \d a}_{I\!I_k},
\end{aligned}
\end{equation*}
where $\delta > 0$ is independent of $k$. Again by dominated convergence, we have $I_k \to 0$. For $I\!I_k$, we write
\begin{equation*}
\begin{aligned}
I\!I_k \lesssim&  \int_{B(0,R)} \int\int_{|x-y| \leq \delta} \frac{|u_k^a(x) - u_k^a(y)|^n + |u^a(x) - u^a(y)|^n}{|x-y|^{2n-1}} \d x \d y \d a\\
\lesssim & \underbrace{ \int\int_{|x-y| \leq \delta}\frac{1}{|x-y|^{2n-1}} \int_{B(0,R)} \min \left(\frac{|u(x)-u(y)|^n}{\min(|u(x)-a|^n,|u(y)-a|^n)}, 2\right)  \d a \d x \d y }_{I\!I\!I}\\
& + \underbrace{ \int\int_{|x-y| \leq \delta}\frac{1}{|x-y|^{2n-1}} \int_{B(0,R)} \min \left(\frac{|u_k(x)-u_k(y)|^n}{\min(|u_k(x)-a|^n,|u_k(y)-a|^n)}, 2\right)  \d a \d x \d y }_{IV_k}
\end{aligned}
\end{equation*}

We note that $I\!I\!I$ is independent of $k$, and the interior integral can be estimated by
\begin{equation*}
  \begin{split}
  \int_{B(0,R)} \min &\left(\frac{|u(x)-u(y)|^n}{\min(|u(x)-a|^n,|u(y)-a|^n)}, 2\right)  \d a\\
  &\lesssim  |u(x)-u(y)|^n (|\log |u(x) -u(y)|| + \log R)\\
  &\leq C(\e)  (1+\log R)\max \left(|u(x)-u(y)|^{n-\e},|u(x)-u(y)|^{n+\e}\right)
\end{split}
\end{equation*}
for  $\e>0$ chosen small enough such that
\[
  \begin{split}
  \frac{n-1}{n-\e}&< s\\
  s-\frac{n}{p}&> -\frac{1}{n+\e}\,.
\end{split}
  \]
  This choice of $\e$ implies
  \[
    W^{s,p}(\Omega;\R^n)\subset W^{\frac{n-1}{n+\eps},n+\eps}\cap W^{\frac{n-1}{n-\eps},n-\eps} (\Omega;\R^n)
    \]
by standard embeddings for fractional Sobolev spaces (see e.g.~\cite{triebel2006theory}, Theorem 3.3.1).
Thus we get
\begin{equation*}
I\!I\!I \leq C(s,p,\Omega,R) \int\int_{|x-y| \leq \delta}\frac{\max\left(|u(x) - u(y)|^{n+\eps},|u(x)-u(y)|^{n-\e}\right)}{|x-y|^{2n-1}}  \d x \d y .
\end{equation*}
As $u \in W^{\frac{n-1}{n+\eps},n+\eps}\cap W^{\frac{n-1}{n-\eps},n-\eps} $, we have $I\!I\!I \to 0$ as $\delta \to 0$.

For $IV_k$,we arrive at the same estimate
\begin{equation*}
IV_k \leq C(s,p,\Omega,R) \int\int_{|x-y| \leq \delta}\frac{\max\left(|u_k(x) - u_k(y)|^{n+\eps},|u_k(x)-u_k(y)|^{n-\e}\right)}{|x-y|^{2n-1}}  \d x \d y .
\end{equation*}
Since $u_k \to u$ strongly in $W^{\frac{n-1}{n+\eps},n+\eps}\cap W^{\frac{n-1}{n-\eps},n-\eps}$, the integrand is compact in $L^1$ and thus uniformly integrable, and we have $\sup_k IV_k \to 0$ as $\delta \to 0$.
\end{proof}

\begin{definition}
  For a distribution $T\in\mathcal D'(\Omega)$, we define
  \[
    \|T\|_{W^{-1,1}(\Omega)} = \sup\{\langle T, \psi \rangle\,:\,\psi\in C_c^\infty(\Omega), \Lip \psi \leq 1\}\,,
  \]
where $\Lip\psi$ denotes the Lipschitz constant of $\psi$,   and we set $W^{-1,1}(\Omega)=\{T\in\mathcal D'(\Omega):\|T\|_{W^{-1,1}(\Omega)}<\infty\}$.
\end{definition}

In the following lemma, we will consider the extension of a function $u\in W^{\crit,n}(\Omega;\R^n)$ to $\Omega\times[0,1)$ that we already used in the proof of Proposition \ref{prop:distrib}: Let $\tilde u\in W^{\crit,n}(\R^n;\R^n)$ be an extension of $u$ to $\R^n$ such that
\[
  \|\tilde u\|_{W^{\frac{n-1}{n},n}(\R^n;\R^n)}\lesssim \|u\|_{W^{\frac{n-1}{n},n}(\Omega;\R^n)}
\]
and define $U\in W^{1,n}(\Omega\times [0,1);\R^n)$ by
\[
  U(x,t)=\eta_t* \tilde u(x)\,.
  \]
The following Lemma then shows a certain continuity of the $Ju^a$ under mollification. A similar statement can be found in \cite[Theorem 1.1]{hang2000remark}.

\begin{lemma}\label{lemma: Ju_k^a limit}
Let $\Omega \subseteq \R^n$ be open, with Lipschitz boundary. Let $u\in W^{\crit,n}(\Omega;\R^n)$, and $U$ the extension to $\Omega\times[0,1)$ defined above.  Let $t_k \downarrow 0$ and take $u_k = U(\cdot,t_k) \in {W^{\crit,n} \cap C^\infty(\Omega;\R^n)}$.
Then there is a family of distributions $(T^a )_{a\in \R^n} \subset W^{-1,1}(\Omega)$ such that
\begin{equation*}
\int_{\R^n} \|J u_k^a - T^a \|_{W^{-1,1}(\Omega)} \d a \to 0.
\end{equation*}
\end{lemma}

% \begin{remark}
% Note that Lemma \ref{lemma: Ju_k^a limit} implies that
% \begin{equation*}
% \int_{\R^n} \|T^a\|_{W^{-1,1}(\Omega)} \d a\leq \|u\|_{W^{\frac{n-1}{n},n}(\Omega;\R^n)}^n\,.
% \end{equation*}
% \end{remark}

% \begin{remark}
% Since $W^{\frac{n-1}{n},n}(\Omega)$ is the trace space of $W^{1,n}(\Omega \times (0,\infty)$, it follows from the inverse trace theorem (see e.g. \cite{Lions1963}) that an extension $U$ always exists.

% In fact, for $\Omega = \R^n$, one such extension is given by setting $U(x,t) = u\ast \phi_t (x)$, where $(\phi_t)_{t>0}$ is a standard mollifier. In that case, $u_k = u\ast \phi_{t_k}$.

% For $\Omega \neq \R^n$, one can first find an extension of $u$ to $W^{\frac{n-1}{n},n}(\R^n)$ and then define $U$ the same way.
% \end{remark}

\begin{remark}
\label{rem:Tconv}
Combining Lemma  \ref{lemma: u^a continuity} and \ref{lemma: Ju_k^a limit}, we
obtain that $T^a = J u^a$ for almost every $a$ if $u\in W^{s,p}(\Omega;\R^n)$, $sp>n-1$, $s>\crit$.
\end{remark}

\begin{proof}[Proof of Lemma \ref{lemma: Ju_k^a limit}]

By the coarea formula for $U$ (see \cite[Theorem 2.93]{AFP}) we have
\begin{equation}\label{eq: coarea}
\int_{\R^n} \H^1(U^{-1}(a))\d a = \int_{\Omega \times [0,1)} |JU|\d x \leq  \|U\|_{W^{1,n}(\Omega\times [0,1);\R^n)}^n.
\end{equation}

In addition, because $U$ is smooth, almost all $a\in \R^n$ are regular values of $U$, so that $U^{-1}(a)\cap (\Omega \times (0,1))$ is a $1$-manifold, i.e. a countable union of curves that are either closed or terminate on either $\Omega \times \{0\}$, on $\Omega \times \{1\}$, or on $\partial \Omega \times [0,1]$.

Now $u_k = U(\cdot,t_k) \in W^{\crit,n}\cap C^\infty(\Omega;\R^n)$. We note that almost every $a\in \R^n$ is a regular value of both $U$ and of all $u_k$.

For such an $a$, $J u_k^a = \omega_n \sum_{x\in u_k^{-1}(a)} \sigma(x) \delta_x$, with $\sigma(x) = \pm 1$, see the appendix. We can then estimate for $t_k \leq t_l$
\begin{equation}\label{eq: Jacobians Cauchy}
\|J u_k^a - J u_l^a\|_{W^{-1,1}(\Omega)} \leq \omega_n \H^1\left(U^{-1}(a) \cap (\Omega \times (t_k, t_l))\right).
\end{equation}

To see this, consider a point $x\in \supp Ju_k^a$. Then there is a unique curve $c\subseteq U^{-1}(a)$ passing through $(x,t_k)$ transversally. We follow this curve into the set $\Omega \times (t_k, t_l)$ until we intersect either
\begin{itemize}
\item[(i)] $\Omega \times \{t_k\}$
  \item[(ii)] or $\Omega \times \{t_l\}$
 \item[(iii)] or $\partial \Omega \times (0,1)$,
\end{itemize}
whichever comes first. See Figure \ref{fig: level sets} for a graphical analogy. In case (i), call the intersection point $(y,t_k)$. Then $y\in \supp Ju_k^a$ and $\sigma(y) = -\sigma(x)$. In case (ii), call the point $(y,t_l)$. Then $y\in \supp Ju_l^a$ and $\sigma(y) = \sigma(x)$. In case (iii), the intersection point is $(y,t)$, with $y\in \partial \Omega$ and $t\in [t_k,t_l]$ (a unique intersection point exists because $c$ has finite length).
For a point $x \in \supp Ju_l^a$ we can also find a corresponding second point $(y,t_k)$ or $(y,t_l)$ or $(y,t)$ as above.

All the points that are covered by the first two cases occur in pairs $(x_i,y_i)$ such that their contributions appear with opposite signs when computing
$Ju_k^a - Ju_l^a$. We collect the corresponding indices $i$ in a set $I'$, and 
the set $I''$ serves as index set of the points occuring in case (iii) above.
Thus we have for some $\sigma_i,\sigma_j \in \{\pm 1\}$
\begin{equation*}
Ju_k^a - Ju_l^a = \omega_n \left( \sum_{i\in I'} \sigma_i(\delta(x_i) - \delta(y_i)) + \sum_{j\in I''} \sigma_j \delta(x_j) \right), 
\end{equation*}
with $\sum_{i \in I'} |x_i - y_i| + \sum_{j\in I''} \dist(x_j, \partial \Omega) \leq \H^1\left(U^{-1}(a) \cap (\Omega \times (t_k, t_l))\right)$. If $\psi\in C_c^1(\R^n)$ is a test function with $\Lip \psi \leq 1$, we have
\begin{equation*}
  \begin{split}
    |\langle Ju_k^a - Ju_l^a, \psi \rangle|& \leq \omega_n \left( \sum_{i\in I'} |\psi(x_i) - \psi(y_i)| + \sum_{j\in I''} |\psi(x_j)| \right)\\
    &\leq \omega_n\H^1\left(U^{-1}(a) \cap (\Omega \times (t_k, t_l))\right),
\end{split}
\end{equation*}
which implies \eqref{eq: Jacobians Cauchy}.

By \eqref{eq: coarea} and \eqref{eq: Jacobians Cauchy}, the maps $a \mapsto J u_k^a$ form a Cauchy sequence in $L^1(\R^n;W^{-1,1}(\Omega))$. The statement follows from the completeness of that space.

\end{proof}

\begin{figure}
\includegraphics{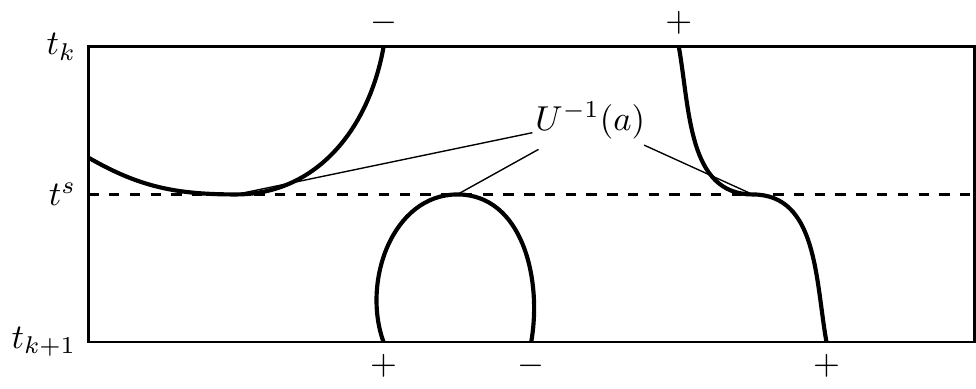}
\caption{The level set $U^{-1}(a)$ for a regular value $a\in \R^n$. Note that $a$ is a singular value of $u_{t^s}$ only for a null set of $t^s\in (0,1)$.}\label{fig: level sets}
\end{figure}

\begin{proof}[Proof of Theorem \ref{theorem: weak coarea chain}]
The weak coarea formula holds in $C^\infty(\Omega;\R^n)$ (see the appendix).
For $u\in W^{s,p}(\Omega;\R^n)$ consider the  extension $U\in W^{1,n}\cap C^\infty(\Omega \times [0,1);\R^n)$ defined above Lemma \ref{lemma: Ju_k^a limit} and set $u_k = U(\cdot,t_k)$ for some sequence $t_k \downarrow 0$.

%Why does this sequence converge in the stronger norm? If we used standard mollifier this would be more clear...

Now $u_k \to u$ in $W^{s,p}(\Omega;\R^n)$, and by Lemma \ref{lemma: u^a continuity} the maps $a\mapsto u_k^a$ converge to $a\mapsto u^a$ in $L^1_{\mathrm{loc}}(\R^n;W^{\crit,n}(\Omega;\R^n))$. Thus $u_k^a \to u^a$ in $W^{\crit,n}(\Omega;\R^n)$ for almost every $a\in \R^n$ after extracting a subsequence. Fix a test function $\psi\in C_c^\infty(\R^n)$. Then
\begin{equation*}
\langle Ju, \psi \rangle = \lim_{k\to \infty} \langle Ju_k, \psi \rangle = \lim_{k\to \infty} \frac{1}{\omega_n} \int_{\R^n} \langle Ju_k^a, \psi \rangle \d a.
\end{equation*}

By Proposition  \ref{prop:distrib} and Lemma \ref{lemma: u^a continuity}, we
have $\langle J u_k^a,\psi\rangle \to \langle J u^a,\psi\rangle$ for almost
every $a\in \R^n$. On the other hand, by Lemma
 \ref{lemma: Ju_k^a limit}, $\int_{\R^n} \|J u_k^a - T^a\|_{W^{-1,1}}\d a \to 0$ for some family of distributions $(T^a)_{a\in \R^n}$. As we have already stated in Remark \ref{rem:Tconv}, it follows that $T^a = Ju^a$ for almost every $a$, and in particular
\begin{equation*}
\lim_{k\to \infty} \frac{1}{\omega_n} \int_{\R^n} \langle Ju_k^a, \psi \rangle \d a = \frac{1}{\omega_n} \int_{\R^n} \langle Ju^a, \psi \rangle \d a.
\end{equation*}

We show the weak chain rule in the same way. It clearly holds for all $u_k$. Since $F\in C^1(\R^n;\R^n)$ is globally Lipschitz, $F \circ u_k \to F \circ u$ in $W^{s,p}(\Omega;\R^n)$, so that
\begin{equation*}
\langle J(F\circ u), \psi \rangle = \lim_{k\to \infty} \langle J(F\circ u_k), \psi \rangle = \lim_{k\to \infty} \frac{1}{\omega_n} \int_{\R^n} \det \nabla F(a) \langle J u_k^a, \psi \rangle \d a.
\end{equation*}

Once again,
\begin{equation*}
\lim_{k\to \infty} \int_{\R^n} \|Ju_k^a -Ju^a\|_{W^{-1,1}(\Omega)} \d a \to 0,
\end{equation*}

and since $\det \nabla F \in L^\infty(\R^n)$, we have
\begin{equation*}
\lim_{k\to \infty} \frac{1}{\omega_n} \int_{\R^n} \det \nabla F(a) \langle Ju_k^a, \psi \rangle \d a =  \frac{1}{\omega_n} \int_{\R^n} \det \nabla F(a) \langle Ju^a, \psi \rangle \d a.
\end{equation*}

\end{proof}

\section{Proof of the strong chain rule}
\label{sec:proof-strong-chain}
\newcommand{\W}{{W^{\frac{n}{n+1},n+1}(\Omega;\R^n)}}
\newcommand{\Wo}{{W^{\frac{n}{n+1},n+1}(\Omega;\R^n)}}

As a consequence of Proposition \ref{prop:distrib}, if $u\in \Wo$, we can extend $Ju$ to a bounded linear functional on $\Wo$:
\begin{corollary}
  For $u\in \W$, we have that $Ju\in \left(\Wo\right)^*$.
\end{corollary}

% Let $F\in C^\infty_c(\R^n;\R^n)$, $u\in W^{n/(n+1),n+1}(U;\R^n)$ such that $[Ju]$
% is a Radon measure.
From the corollary, we immediately obtain the proof of Theorem 
\ref{thm:strong}:

\begin{proof}[Proof of Theorem \ref{thm:strong}]
For a smooth approximation
$u_k\to u$ in $W^{\frac{n}{n+1},n+1}(\Omega;\R^n)$, we have
\[
  \det DF\circ u_k\to \det DF\circ u\quad\text{  in }W^{\frac{n}{n+1},n+1}(\Omega)
\]
as well as $F\circ u_k\to F\circ u$ in $W^{\frac{n}{n+1},n+1}(\Omega;\R^n)$. (For our present purpose, it would be enough to have the latter convergence in $W^{\crit,n}(\Omega;\R^n)$.) Thus, for every $\psi\in C^1_c(\Omega)$,  we have that 
\begin{equation}
  \begin{split}
  \left<J(F\circ u),\psi\right>&=\lim_{k\to \infty}\left<J(F\circ u_k),\psi\right>\\
  &=\lim_{k\to \infty}\left<Ju_k,\psi (\det DF)\circ u_k\right>\\
&=    \left<Ju,\psi (\det DF)\circ u\right> \,.\label{eq:7}
\end{split}
\end{equation}
In other words, we have $J(F\circ u)=Ju(\det DF)\circ u$ as Radon measures.
This is just the 
claim (ii) of Theorem \ref{thm:strong}.

\medskip

By Theorem 13 in \cite{de2003some}, it follows that the strong coarea formula
holds too. One needs to note that  the assumptions of that Theorem include
\[
  u\in L^\infty\cap W^{1,n-1} \quad \text{ and } \quad Ju \text{ is a Radon
    measure,}
  \]
but in the proof only the latter condition is used; it goes through for $u\in
C^0\cap W^{\frac{n}{n+1},n+1}(\Omega;\R^n)$. Also,
the requirement that  \eqref{eq:7} holds for
every $F\in C^1_c(\R^n;\R^n)$ (as required in the statement of the quoted
theorem) can be replaced by its validity for every $F\in
C^\infty_c(\R^n;\R^n)$, since the proof in \cite{de2003some} works by approximation,
and that is possible even with the weaker requirement.
\end{proof}

\section{Chain rules for H\"older functions}
\label{sec:chain-rules-holder}
\subsection{Proof of the chain rule for H\"older functions}
\label{sec:proof-chain-rule-hoelder}
Let $\Omega\subset \R^n$ be as before, and $\alpha\in (\crit,1)$. The H\"older space
$C^{0,\alpha}(\Omega;\R^n)$ is  a subset of $W^{\crit,n}(\Omega;\R^n)$ (see
\cite{triebel2006theory}), and thus the weak chain rule holds. Slightly more than
that can
be said in this case:
We have already noted in Definition \ref{def:Ju} that for $u\in C^{0,\alpha}$,
$Ju$ is an element of the dual of $BV$.
% As a consequence of Proposition \ref{prop:distrib}, we have
% \begin{corollary}
%   \label{lem:estim_jac_holder}
%   Let $\alpha\in(\crit,1)$, $u,v\in C^{0,\alpha}(\Omega;\R^n)$ and $ \psi\in BV(\Omega)$. Then
%   \[
% \left<Ju-Jv,\psi\right>\leq
%     C\left(\|u\|_{C^{0,\alpha}}+\|v\|_{C^{0,\alpha}}\right)^{n-1}\|u-v\|_{C^{0,\alpha}}|D\psi|(\Omega)\,,
%   \]
%   where $C$ depends only on $\Omega$ and $\alpha$.
%   As a consequence, the map $u\mapsto Ju$ is continuous from
%   $C^{0,\alpha}(\Omega;\R^n)$ to the dual of $BV(\Omega)$.
% \end{corollary}
Our aim is now to extend the weak chain rule to the case $\psi\in
BV(\Omega)$. Since $u^a$ is not a H\"older function, we cannot
immediately make sense of the right hand side in that equation. However if we
fix 
the test function to be the characteristic function of a Lipschitz set 
$E$, we have the following:
\begin{lemma}
 \label{lem:estim_bdry_holder} 
Let $u,v\in C^{0,\alpha}(\Omega;\R^n)$, and $a\in\R^n$. Then $\mathbf j u^a\in C^0(\Omega\setminus
u^{-1}(a);\R^n)$ and $\mathbf j v^a\in C^0(\Omega\setminus
v^{-1}(a);\R^n)$, and for every Lipschitz set $E$ with
   $\overline E\subset\Omega$ and $a\not\in u(\partial E)\cup v(\partial E)$, we
   have that
  \[
    \left<J u^a-J v^a,\chi_E\right>\leq
    C\left(\|u\|_{C^{0,\alpha}}+\|v\|_{C^{0,\alpha}}\right)^{n-1}\|u-v\|_{C^{0,\alpha}}\mathrm{Per}
    (E)\,,
  \]
  where $\mathrm{Per}(E)=|D\chi_E|(\Omega)$ denotes the perimeter of $E$ and the constant $C$ depends only on $\dist(a,u(\partial E))$,
  $\dist(a,v(\partial E))$.
\end{lemma}

\begin{proof}
  Choose $r>0$ such that $\min(\dist(a,u(\partial E)),\dist(a,v(\partial E)))>2r$,
  set
  \[
    \Omega_r \coloneqq \Omega\setminus\left(u^{-1}(B(a,r))\cup v^{-1}(B(a,r))\right),
    \]
  and apply the  previous corollary to the function $u^a\in C^{0,\alpha}(\Omega_r;\R^n)$. The claim follows from the elementary inequality
    \[
      \|u^a\|_{C^{0,\alpha}(\Omega_r;\R^n)}\lesssim r^{-1}\|u\|_{C^{0,\alpha}(\Omega;\R^n)}\,.
      \]
  \end{proof}

For a given Lipschitz set $E$, we have that $\L^n(u(\partial E))=0$ (see
e.g.~\cite{olbermann2015integrability} Lemma 2.7), and
hence $\left<Ju^a,\chi_E\right> = -\left<\mathbf j u^a,D\chi_E\right>$ is
defined for almost every $a\in\R^n$. In fact, the weak coarea formula and chain rule hold:

\begin{proof}[Proof of Theorem \ref{thm:hoelder}]
  Choose $\tilde \Omega$ such that $E\Subset\tilde\Omega\Subset\Omega$. 
  For $\e>0$ small enough, we have that $\dist(\tilde \Omega,\partial\Omega)>\e$ and we may
  assume that  $u_\e \coloneqq u*\eta_\e$ is defined in $\tilde \Omega$. Choose $\tilde
  \alpha\in (\crit,\alpha)$. With this choice we have that 
   $u_\e\to u$, $F\circ u_\e\to F\circ u$ in $C^{0,\tilde \alpha}(\tilde
   \Omega;\R^n)$, and
   \[
     % \left<[Ju_\e],\chi_E\right> \to \left<[Ju],\chi_E\right>\,,\qquad
     \left<J(F\circ u_\e),\chi_E\right> \to \left<J(F\circ u),\chi_E\right>\,.
   \]
    Now for $\e>0$, $u_\e$ is smooth, which
   means that by classical change of variables formula we obtain
   \[
     \begin{split}
       % \left<[Ju_\e],\chi_E\right>&=\int_{E}\det\nabla u_\e\d x\\
       % &=\int_{\R^n}\deg(u_\e,E,a)\d a\,\,,\\
       \left<J(F\circ u_\e),\chi_E\right>&=\int_{E}\det\nabla F(u_\e(x))\det\nabla u_\e(x)\d x\\
       &=\int_{\R^n}\det\nabla F(a)\deg(u_\e,E,a)\d a\,\,.
     \end{split}
   \]  
   Supposing $a\not \in u_\e(\partial E)$, we have % by equation \eqref{eq:1} in
   % the appendix
   that
   \[
     \begin{split}
       \deg(u_\e, E ,a)&=\frac{1}{\omega_n}\int_{\partial E }(u_\e)^a \cdot \cof \nabla (u_\e)^a \nu\d\H^{n-1}\\
       &=\frac{1}{\omega_n} \left<J(u_\e)^a,\chi_E\right>\,,
     \end{split}
   \]
   where $\nu$ denotes the unit outer normal of $\partial E $, see the
   appendix. By Theorem 1.2 in \cite{olbermann2015integrability}, we have that
   $\deg(u_\e, E ,\cdot)$ is a convergent sequence in $L^1(\R^n)$, and by Lemma
   \ref{lem:estim_bdry_holder} we have that
   \[
     \left<J(u_\e)^a,\chi_E\right>\to \left<Ju^a,\chi_E\right>
   \]
   for almost every $a\in \R^n$. This implies that
   \[
     \begin{split}
       % \int_{\R^n}\left<[J(u_\e)^a],\chi_E\right>\d a&\to \int_{\R^n}\left<[J
       %   u^a],\chi_E\right>\d a\,\\
              \int_{\R^n}\det\nabla F(a)\left<J (u_\e)^a,\chi_E\right>\d a&\to
              \int_{\R^n}\det\nabla F(a)\left<J
         u^a,\chi_E\right>\d a\,,
     \end{split}
 \]
 proving our claim.
\end{proof}

\begin{remark}
  For $\alpha\in (\frac{n}{n+1},1)$, the strong chain rule holds for $u\in
  C^{0,\alpha}(\Omega;\R^n)$ by the
  inclusion $C^{0,\alpha}\subset W^{\frac{n}{n+1},n+1}$.
\end{remark}

\subsection{Two  open questions linked to the validity of the strong chain
  rule}
\label{sec:two-open-questions}
A very important  and completely open question is whether or not the strong  chain rule holds for functions $u\in C^0\cap W^{s,p}(\Omega;\R^n)$
      with $s\in [\frac {n-1}n,\frac{n}{n+1})$ and $sp\in[n-1,n)$ such that
      $Ju$ can be represented by a non-negative Radon measure.
A positive answer to this question would prove in particular two 
open problems of high profile, concerning the rigidity of $C^{1,\alpha}$ isometric immersions and
the existence of $C^{0,\alpha}$ embeddings of the two-dimensional disk into  the Heisenberg group. In the
present section, we want to briefly explain the open problems and their
relation to the strong chain rule. 

\medskip

\subsubsection{The $C^{1,\alpha}$ Weyl problem} Let $g$ be a (smooth) Riemannian metric with positive curvature on $S^2$. By the
Nash-Kuiper theorem, any short immersion $f:S^2\to\R^3$ can be approximated
arbitrarily well in $C^0$ by an isometric immersion $\bar f\in
C^1(S^2;\R^3)$. In particular, there exists a very large set of solutions to the
isometric immersion problem in the class of  $C^1$-immersions. This is in
stark contrast to the situation when one requires higher regularity: By a result
by Pogorelov \cite{MR0346714}, there exists a unique solution in the class of $C^2$-immersions
(up to Euclidean motions). In other words, $C^2$ isometric immersions are \emph{rigid}.

Concerning  spaces between $C^1$ and $C^2$, Borisov has shown in a series of
works \cite{MR0104277,MR0104278,MR0116295,MR0131225,MR0192449,MR2047871} that
$C^{1,\alpha}$ isometric immersions are rigid if $\alpha>\frac23$, see also
\cite{conti2012h} for a much shorter proof. On the other side, it has
been shown that for $\alpha<\frac{1}{5}$, there exists again a very large set of
$C^{1,\alpha}$ (local) isometric immersions  \cite{MR3850282}. The question
whether or not $C^{1,\alpha}$ isometric immerisons are rigid in the parameter
range $\alpha\in[\frac15,\frac23)$ is open.

\medskip

The rigidity proof in \cite{conti2012h} is   based on proving that the
immersed surfaces are of \emph{bounded extrinsic curvature}. The core of the
proof consists in  showing the following: Let $U\subset \R^2$ be a coordinate
chart of $S^2$ with Lipschitz boundary, and $\psi\in C^1_c(U)$. Let $\nu:U\to
\R^2$ be a representation of the surface normal in coordinate charts. The
boundedness of extrinsic curvature follows from the identity
  \begin{equation}
    \label{eq:5}
    \int_{U} \psi(\nu(x))\kappa_g(x) \d
  A_g(x)=\int_{\R^2}\psi(y)\deg(\nu,U,y)\d y\,,
\end{equation}
 where   $\kappa_g$ is the Gauss curvature, and $\d A_g$ the surface element of the
 manifold $(S^2,g)$ in the coordinate chart $U$. By the positivity of Gauss
 curvature, the above identity allows for an estimate of the extrinsic curvature. 

 \medskip
 
Let  $f^\psi=(f^\psi_1,f^\psi_2)$ denote a
solution of $\div
f^\psi=\psi$, and let us assume that $f^\psi\in C^1(U;\R^2)$. Let $F^\psi_1(x_1,x_2)=(f^\psi_1(x_1,x_2),x_2)$ and
$F^\psi_2(x_1,x_2)=(x_1,f^\psi_2(x_1,x_2))$. 
Now suppose that
$\nu\in C^{0,\alpha}(\Omega;\R^2)$ with $\alpha>\frac12$.
Note that $C^{0,\alpha}\subset W^{1/2,2}$. 
By the weak chain rule,  
\[
\sum_{i=1}^2\left<J(F^\psi_i\circ \nu),\chi_U\right>=\int_{\R^2}\psi(y)\deg(\nu,U,y)\d y\,.
\]
For $\alpha> \frac23$, we have by the strong chain rule that
\[
 \sum_i J(F^\psi_i\circ \nu)=\sum_i\det DF^\psi_i(\nu)J\nu\,.
\]
Testing this equation with $\chi_U$ gives precisely 
\eqref{eq:5}, and the boundedness of extrinsic curvature follows. If one were
able to show the strong chain rule  also for $\alpha>\frac12=\frac{n-1}{n}$,
the same would be true for this larger parameter range.

\subsubsection{The H\"older mapping problem}
The Heisenberg group $\mathbb H$ can be thought of as one of the simplest examples of a
sub-Riemannian manifold. As a metric space, it can be defined as the pair
$(\R^3,d_{CC})$, where the so-called Carnot-Carath\'eodory distance $d_{CC}$ is defined
as follows: For $p\in\R^3$, define the one-form $\Theta_p=\d x_3+\frac12(x_1\d x_2-x_2\d
x_1)$. A Lipschitz curve $\gamma$ in $\R^3$ is said to be  admissible if
$\Theta(\dot\gamma)=0$ almost everywhere. Then 
$d_{CC}$ is defined by taking the infimum of lengths of  admissible curves,
\[
  d_{CC}(p,q)=\inf\{\mathrm{length}(\gamma)|\,\gamma:[0,T]\to\R^3 \text{
    admissible, }\gamma(0)=p, \,\gamma(T)=q\}\,.
  \]
It can be shown that the so-called Kor\'anyi  metric $d_K$ is Bi-Lipschitz equivalent to
$d_{CC}$,
\[
  d_K(p,q)^4=\left((p_1-q_1)^2+(p_2-q_2)^2\right)^2+\left(p_3-q_3+p_1q_2-p_2q_1\right)^2\,.
    \]
The H\"older mapping problem is to find a $C^{0,\alpha}$ embedding of the
two-dimensional disk into $\mathbb H=(\R^3,d_K)$. It has been shown by Gromov that
this is impossible for $\alpha>\frac23$. On the other hand, any smooth embedding
with respect to the Euclidean distance in $\R^3$ is a  $C^{0,1/2}$ embedding
with with respect to $d_{CC}$. Existence for the range $\alpha\in
(\frac12,\frac23]$ is an open question.

\medskip

The argument for  non-existence if $\alpha>\frac23$ is as follows: If
$\gamma:[0,T]\to\mathbb H$ is a $C^{0,\alpha}$ curve with $\alpha>\frac12$, then
the $x_3$-component of $\gamma$ is completely determined by the $x_1$ and $x_2$
components,
\[
  \gamma_3(t)-\gamma_3(0)=\frac12\int_0^t\left(\gamma_1\d\gamma_2-\gamma_2\d\gamma_1\right)\,,
\]
where the right hand side has to be understood as a Lebesgue-Stieltjes
integral (see Lemma 3.1 in \cite{le2013some} and references therein). It follows
that for $U\subset\R^2$ and a set $\Gamma\subset U$ that is Bi-Lipschitz
equivalent to $S^1$, and $v\in C^{0,\alpha}(U;\mathbb H)$ we have that
  \begin{equation}
    \label{eq:2}
  \frac12\int_\Gamma\left(v_1\d v_2-v_2\d v_1\right)=0\,.
\end{equation}

On the other hand, one can show the following: Let $\pi:\mathbb H\to\R^2$ denote
the projection onto the first two components. If $U\subset\R^2$ is open, and if
$v:U\to\mathbb H$ is a $C^{0,\alpha}$ embedding
with $\alpha>\frac12$, then for every open subset $V\subset U$, there exists a
closed Lipschitz curve $\Gamma$ in $V$ such that the curve $\pi\circ v\circ\Gamma$ defines
a non-vanishing current, see Lemma 3.3 in \cite{le2013some}. We reformulate the
latter:
Let $W$ be the union of the bounded components of $\R^2\setminus\Gamma$. Let
$u \coloneqq \pi\circ v$. Then the conclusion of Lemma 3.3 in \cite{le2013some}  can be rephrased by saying that 
there exists a function $\psi\in C^1_c(\R^2\setminus u(\partial W))$ such that
  \begin{equation}
\label{eq:4}
    \int_{\R^2}\psi(y)\deg(u,W,y)\d y\neq 0\,.
\end{equation}

As above, we may rewrite $\psi=\sum_{i=1}^2 \det\nabla F^\psi_i$, and we obtain by
the weak chain rule 
  \begin{equation}
    \label{eq:3}
  \int_{\R^2}\psi(y)\deg(u,W,y)\d y= \sum_{i=1}^2 \left<J (F^\psi_i\circ
    u),\chi_W\right>\,.
\end{equation}
Note that  \eqref{eq:2}
implies  
that $Ju=0$.  For $\alpha>\frac23$ the strong chain rule
holds; hence \eqref{eq:4} and \eqref{eq:3} form a contradiction. Again, a proof of the strong
chain rule in the range $\alpha\in (\frac12,\frac23]$ would immediately imply
the result for the larger range.
  
\appendix

\section{The weak coarea formula and chain rule for smooth functions}

For $u\in C^1(\Omega;\R^n)$ let $\deg(u,\Omega,\cdot)$ denote the Brouwer degree,
defined at regular points $y\in \R^n\setminus u(\partial\Omega)$ by
\[
\deg(u,\Omega,y) \coloneqq \sum_{x\in u^{-1}(y)}\sgn \det \nabla u(x)\,.
\]
It is well known that for $\psi\in C^1_c(\R^n)$, we have the change of
variables formula
\[
\int_{\R^n}\deg(u,\Omega,y)\psi(y)\d y=\int_\Omega \psi(u(x))\det \nabla
u(x)\d x\,.
\]
The Brouwer degree is constant on the connected components of $\R^n\setminus u(\partial\Omega)$.
For $a\in \R^n\setminus u(\partial\Omega)$, choose $\e_a$ such that
$\dist(a,u(\partial\Omega))>\e_a$. We may choose the test function $\psi_a \in C_c^0(\R^n)$ such that $\int_{\R^n}\psi_a(y)\d y=1$,
$\psi_a=\div V_a$ for some $V_a\in C^1(\R^n;\R^n)$ with $V_a(y)=\frac{1}{n\omega_n} \frac{y-a}{|y-a|^n}$ for $y\in\R^n\setminus B(a,\e_a)$. The
 latter implies $\psi_a=0$ on $\R^n\setminus B(a,\e_a)$, and 

  \begin{equation}
\begin{split}
  \deg(u,\Omega,a)&= \int_{\R^n} \deg(u,\Omega,y)\psi_a(y)\d y\\
  &=\int_{\Omega}(\div V_a)(u(x))\det \nabla u(x)\d x\\
  &=\int_{\Omega} \div ((\cof \nabla u(x))^TV_a(u(x)))\d x\\
& =\frac{1}{n \omega_n}\int_{\partial \Omega} (\cof \nabla u(x))^T\frac{u(x)-a}{|u(x)-a|^n}\cdot  \nu(x)\d\H^{n-1}(x)\\
& =\frac{1}{\omega_n}\int_{\partial \Omega} \j u^a(x) \cdot \nu(x)\,\d \H^{n-1}(x)\,,
\end{split}\label{eq:1}
\end{equation}

where $\nu$ denotes the unit outer normal of $\Omega$ and $u^a(x)=\frac{u(x)-a}{|u(x)-a|}$.

Now let $\psi\in C^1_c(\Omega)$, and assume first that $\psi\geq 0$. We write
\[
  E_t \coloneqq \{x\in\Omega:\psi(x)>t\}
\]
and have that $\psi(x)=\int_0^\infty\chi_{E_t}(x)\d t$ for every $x\in\Omega$. Using the above, Fubini's Theorem and the $BV$ coarea formula, we get
\[
  \begin{split}
  \langle Ju,\psi\rangle &= \int_{\Omega}\psi(x)\det\nabla u(x)\d x\\
  &=\int_{\Omega}\int_0^\infty\chi_{E_t}(x) \det\nabla u(x)\d t\d x\\
  &=\int_{\R^n}\int_0^\infty \deg(u,E_t,a)\d t\d a\\
  &=\frac{1}{\omega_n}\int_{\R^n}\int_0^\infty \int_{\partial E_t} \j u^a(x) \cdot \nu(x)\,\d \H^{n-1}(x)\d t\d a\\
&  =-\frac{1}{\omega_n}\int_{\R^n}\int_0^\infty \int_{\Omega} \j u^a(x) \cdot \d(D\chi_{E_t})(x)\d t\d a\\
&=-\frac{1}{\omega_n}\int_{\R^n}\int_{\Omega} \j u^a(x) \cdot \nabla\psi(x)\d x\d a\\
&= \frac{1}{\omega_n} \int_{\R^n} \langle J u^a, \psi \rangle \d a \, .
\end{split}
\]
This is the weak coarea formula. The general case (without the restriction $\psi\geq 0$) is obtained by decomposing into a non-negative and a non-positive part,  $\psi=\max(\psi,0)+\min(\psi,0)$, and noting that the  manipulations above still go through for both  parts of $\psi$ (even though they may not be $C^1$). The weak chain rule is proved in the same way:
\[
  \begin{split}
  \langle J(F\circ u),\psi\rangle &= \int_{\Omega}\psi(x)\det\nabla F(u(x))\det\nabla u(x)\d x\\
  &=\int_{\Omega}\int_0^\infty\chi_{E_t}(x) \det\nabla F(u(x))\det\nabla u(x)\d t\d x\\
  &=\int_{\R^n}\int_0^\infty\det\nabla F(a) \deg(u,E_t,a)\d t\d a\\
  % &=\omega_n^{-1}\int_{\R^n}\int_0^\infty \int_{\partial E_t}\det\nabla F(a)  u^a \cof \nabla u^a\cdot \nu\,\d \H^{n-1}\d t\d a\\
% &  =\omega_n^{-1}\int_{\R^n}\int_0^\infty \int_{\Omega}\det\nabla F(a)  u^a \cof \nabla u^a\cdot \d(D\chi_{E_t})\d t\d a\\
&=-\frac{1}{\omega_n}\int_{\R^n}\int_{\Omega}\det\nabla F(a)  \j u^a(x) \cdot \nabla\psi(x)\d x\d a\\
&= \frac{1}{\omega_n} \int_{\R^n} \det \nabla F(a) \langle J u^a, \psi \rangle \d a \,.
\end{split}
\]

\section*{Acknowledgments}
This work has been supported by  Deutsche Forschungsgemeinschaft (DFG, German
Research Foundation) as part of project 350398276.

\bibliographystyle{alpha}
\bibliography{chain}

\begin{thebibliography}{CDLSJ12}

\bibitem[AFP00]{AFP}
L.~Ambrosio, N.~Fusco, and D.~Pallara.
\newblock {\em Functions of bounded variation and free discontinuity problems}.
\newblock Oxford Mathematical Monographs. The Clarendon Press, Oxford
  University Press, New York, 2000.

\bibitem[BN11]{MR2810795}
H.~Brezis and H.-M. Nguyen.
\newblock The {J}acobian determinant revisited.
\newblock {\em Invent. Math.}, 185(1):17--54, 2011.

\bibitem[Bor58a]{MR0104277}
J.~F. Borisov.
\newblock The parallel translation on a smooth surface. {I}.
\newblock {\em Vestnik Leningrad. Univ.}, 13(7):160--171, 1958.

\bibitem[Bor58b]{MR0104278}
J.~F. Borisov.
\newblock The parallel translation on a smooth surface. {II}.
\newblock {\em Vestnik Leningrad. Univ.}, 13(19):45--54, 1958.

\bibitem[Bor59]{MR0116295}
J.~F. Borisov.
\newblock On the connection bewteen the spatial form of smooth surfaces and
  their intrinsic geometry.
\newblock {\em Vestnik Leningrad. Univ.}, 14(13):20--26, 1959.

\bibitem[Bor60]{MR0131225}
J.~F. Borisov.
\newblock On the question of parallel displacement on a smooth surface and the
  connection of space forms of smooth surfaces with their intrinsic geometries.
\newblock {\em Vestnik Leningrad. Univ.}, 15(19):127--129, 1960.

\bibitem[Bor65]{MR0192449}
J.~F. Borisov.
\newblock {$C^{1,\,\alpha }$}-isometric immersions of {R}iemannian spaces.
\newblock {\em Dokl. Akad. Nauk SSSR}, 163:11--13, 1965.

\bibitem[Bor04]{MR2047871}
J.~F. Borisov.
\newblock Irregular surfaces of the class {$C^{1,\beta}$} with an analytic
  metric.
\newblock {\em Sibirsk. Mat. Zh.}, 45(1):25--61, 2004.

\bibitem[CDLSJ12]{conti2012h}
S.~Conti, C.~De~Lellis, and L.~Sz{\'e}kelyhidi~Jr.
\newblock {$h$}-principle and rigidity for {$C^{1,\alpha}$} isometric
  embeddings.
\newblock In {\em Nonlinear Partial Differential Equations}, pages 83--116.
  Springer, 2012.

\bibitem[DL03]{de2003some}
C.~De~Lellis.
\newblock Some remarks on the distributional {J}acobian.
\newblock {\em Nonlinear Analysis: Theory, Methods \& Applications},
  53(7):1101--1114, 2003.

\bibitem[DLISJ18]{MR3850282}
C.~De~Lellis, D.~Inauen, and L.~Sz\'{e}kelyhidi~Jr.
\newblock A {N}ash--{K}uiper theorem for {$C^{1,1/5-\delta}$} immersions of
  surfaces in 3 dimensions.
\newblock {\em Rev. Mat. Iberoam.}, 34(3):1119--1152, 2018.

\bibitem[HL00]{hang2000remark}
Feng-Bo Hang and Fang-Hua Lin.
\newblock A remark on the jacobians.
\newblock {\em Communications in Contemporary Mathematics}, 2(1):35--46, 2000.

\bibitem[JS02]{MR1911049}
R.~L. Jerrard and H.~M. Soner.
\newblock Functions of bounded higher variation.
\newblock {\em Indiana Univ. Math. J.}, 51(3):645--677, 2002.

\bibitem[LDZ13]{le2013some}
E.~Le~Donne and R.~Z{\"u}st.
\newblock Some properties of {H\"o}lder surfaces in the {H}eisenberg group.
\newblock {\em Illinois Journal of Mathematics}, 57(1):229--249, 2013.

\bibitem[Lio63]{Lions1963}
J.-L. Lions.
\newblock Th\'{e}or\`emes de trace et d'interpolation. {IV}.
\newblock {\em Math. Ann.}, 151:42--56, 1963.

\bibitem[Lun18]{lunardi2018interpolation}
A.~Lunardi.
\newblock {\em Interpolation theory}, volume~16.
\newblock Springer, 2018.

\bibitem[Olb16]{olbermann2015integrability}
H.~Olbermann.
\newblock Integrability of the {B}rouwer degree for irregular arguments.
\newblock {\em Ann. Inst. H. Poincar{\'e} Anal. Non-lin{\'e}aire}, 2016.
\newblock Accepted for publication.

\bibitem[Pog73]{MR0346714}
A.~V. Pogorelov.
\newblock {\em Extrinsic geometry of convex surfaces}.
\newblock American Mathematical Society, Providence, R.I., 1973.
\newblock Translated from the Russian by Israel Program for Scientific
  Translations, Translations of Mathematical Monographs, Vol. 35.

\bibitem[Tar07]{MR2328004}
L.~Tartar.
\newblock {\em An introduction to {S}obolev spaces and interpolation spaces},
  volume~3 of {\em Lecture Notes of the Unione Matematica Italiana}.
\newblock Springer, Berlin; UMI, Bologna, 2007.

\bibitem[Tri06]{triebel2006theory}
H.~Triebel.
\newblock Theory of function spaces.
\newblock {\em BirkhauserVerlag, Basel}, 2006.

\end{thebibliography}

\end{document}